\documentclass[oneside,english,reqno]{amsart}
\usepackage[T1]{fontenc}
\usepackage[latin9]{inputenc}
\usepackage{babel}
\usepackage{refstyle}
\usepackage{mathrsfs}
\usepackage{mathtools}
\usepackage{enumitem}
\usepackage{amstext}
\usepackage{amsthm}
\usepackage{amssymb}
\usepackage[all]{xy}
\PassOptionsToPackage{normalem}{ulem}
\usepackage{ulem}
\usepackage[unicode=true,pdfusetitle,
 bookmarks=true,bookmarksnumbered=false,bookmarksopen=false,
 breaklinks=false,pdfborder={0 0 1},backref=false,colorlinks=false]
 {hyperref}
\hypersetup{
 colorlinks=true,citecolor=blue,linkcolor=blue,linktocpage=true}

\makeatletter

\AtBeginDocument{\providecommand\secref[1]{\ref{sec:#1}}}
\AtBeginDocument{\providecommand\defref[1]{\ref{def:#1}}}
\AtBeginDocument{\providecommand\lemref[1]{\ref{lem:#1}}}
\AtBeginDocument{\providecommand\thmref[1]{\ref{thm:#1}}}
\AtBeginDocument{\providecommand\exaref[1]{\ref{exa:#1}}}
\RS@ifundefined{subsecref}
  {\newref{subsec}{name = \RSsectxt}}
  {}
\RS@ifundefined{thmref}
  {\def\RSthmtxt{theorem~}\newref{thm}{name = \RSthmtxt}}
  {}
\RS@ifundefined{lemref}
  {\def\RSlemtxt{lemma~}\newref{lem}{name = \RSlemtxt}}
  {}

\numberwithin{equation}{section}
\numberwithin{figure}{section}
\numberwithin{table}{section}
\theoremstyle{plain}
\newtheorem{thm}{\protect\theoremname}[section]
  \theoremstyle{definition}
  \newtheorem{defn}[thm]{\protect\definitionname}
  \theoremstyle{remark}
  \newtheorem{rem}[thm]{\protect\remarkname}
  \theoremstyle{plain}
  \newtheorem{lem}[thm]{\protect\lemmaname}
  \theoremstyle{plain}
  \newtheorem{cor}[thm]{\protect\corollaryname}
  \theoremstyle{plain}
  \newtheorem{prop}[thm]{\protect\propositionname}
  \theoremstyle{definition}
  \newtheorem{example}[thm]{\protect\examplename}
  \theoremstyle{remark}
  \newtheorem*{acknowledgement*}{\protect\acknowledgementname}

\allowdisplaybreaks
\usepackage{needspace}

\newref{lem}{refcmd={Lemma \ref{#1}}}
\newref{thm}{refcmd={Theorem \ref{#1}}}
\newref{cor}{refcmd={Corollary \ref{#1}}}
\newref{sec}{refcmd={Section \ref{#1}}}
\newref{sub}{refcmd={Section \ref{#1}}}
\newref{subsec}{refcmd={Section \ref{#1}}}
\newref{chap}{refcmd={Chapter \ref{#1}}}
\newref{prop}{refcmd={Proposition \ref{#1}}}
\newref{exa}{refcmd={Example \ref{#1}}}
\newref{tab}{refcmd={Table \ref{#1}}}
\newref{rem}{refcmd={Remark \ref{#1}}}
\newref{def}{refcmd={Definition \ref{#1}}}
\newref{fig}{refcmd={Figure \ref{#1}}}

\setlist[enumerate]{itemsep=5pt,topsep=3pt}
\setlist[enumerate,1]{label=(\roman*),ref=\roman*}
\setlist[enumerate,2]{label=(\alph*),ref=\theenumi \alph*}

\AtBeginDocument{
  
}

\makeatother

  \providecommand{\acknowledgementname}{Acknowledgement}
  \providecommand{\corollaryname}{Corollary}
  \providecommand{\definitionname}{Definition}
  \providecommand{\examplename}{Example}
  \providecommand{\lemmaname}{Lemma}
  \providecommand{\propositionname}{Proposition}
  \providecommand{\remarkname}{Remark}
\providecommand{\theoremname}{Theorem}

\begin{document}

\title[Metric duality: positive definite kernels vs boundary processes]{Metric duality between positive definite kernels and boundary processes}

\author{Palle Jorgensen and Feng Tian}

\address{(Palle E.T. Jorgensen) Department of Mathematics, The University
of Iowa, Iowa City, IA 52242-1419, U.S.A. }

\email{palle-jorgensen@uiowa.edu}

\urladdr{http://www.math.uiowa.edu/\textasciitilde{}jorgen/}

\address{(Feng Tian) Department of Mathematics, Hampton University, Hampton,
VA 23668, U.S.A.}

\email{feng.tian@hamptonu.edu}


\begin{abstract}
We study representations of positive definite kernels $K$ in a general
setting, but with view to applications to harmonic analysis, to metric
geometry, and to realizations of certain stochastic processes. Our
initial results are stated for the most general given positive definite
kernel, but are then subsequently specialized to the above mentioned
applications. Given a positive definite kernel $K$ on $S\times S$
where $S$ is a fixed set, we first study families of factorizations
of $K$. By a factorization (or representation) we mean a probability
space $\left(B,\mu\right)$ and an associated stochastic process indexed
by $S$ which has $K$ as its covariance kernel. For each realization
we identify a co-isometric transform from $L^{2}\left(\mu\right)$
onto $\mathscr{H}\left(K\right)$, where $\mathscr{H}\left(K\right)$
denotes the reproducing kernel Hilbert space of $K$. In some cases,
this entails a certain renormalization of $K$. Our emphasis is on
such realizations which are minimal in a sense we make precise. By
minimal we mean roughly that $B$ may be realized as a certain $K$-boundary
of the given set $S$. We prove existence of minimal realizations
in a general setting. 
\end{abstract}

\maketitle
\tableofcontents{}

\section{\label{sec:Intro}Introduction}

A variety of notions of \textquotedblleft boundary\textquotedblright{}
and boundary representation for general classes of positive definite
kernels are established in \cite{zbMATH06664785}. It allows us to
carry over results and notions from classical harmonic analysis on
the disk to this wider context (see \cite{MR1655831,MR1667822,MR1655832}).

More specifically, starting with a given positive definite (p.d.)
kernel $K$ on $S\times S$, we introduce generalized boundaries for
the set $S$ that carry $K$. It is a measure theoretic ``boundary\textquotedblright{}
in the form of a probability space, but it is not unique. The set
of measure boundaries will be denoted $\mathcal{M}\left(K\right)$.
Indeed, there exists such a generalized boundary probability space
associated to any p.d. kernel. For example, as an element in $\mathcal{M}\left(K\right)$,
we can take a \textquotedblleft measure\textquotedblright{} boundary
to be the Gaussian process having $K$ as its covariance kernel. This
exists by Kolmogorov's consistency theorem. 

The p.d. kernels include those defined on infinite discrete sets,
for example sets of vertices in electrical networks, or discrete sets
which arise from sampling operations performed on p.d. kernels in
a continuous setting, and with the sampling then referring to suitable
discrete subsets. See, e.g., \cite{MR2884231,MR3285408,MR3541255}.

The purpose of the present paper is to study a metric duality between
(\ref{enu:as1}) and (\ref{enu:as2}) below, where 
\begin{enumerate}[label=(\Roman{enumi}),ref=\Roman{enumi}]
\item \label{enu:as1}$K:S\times S\longrightarrow\mathbb{C}$ is a given
positive definite (p.d.) kernel defined on a fixed set $S$, i.e.,
for $\forall N\in\mathbb{N}$, $\forall\left\{ s_{i}\right\} _{i=1}^{N}$,
$s_{i}\in S$, $\forall\left\{ \xi_{i}\right\} _{i=1}^{N}$, $\xi_{i}\in\mathbb{C}$,
we have 
\begin{equation}
\sum_{i}\sum_{j}\xi_{i}\overline{\xi}_{j}K\left(s_{i},s_{j}\right)\geq0;\;\text{and}\label{eq:a1}
\end{equation}
\item \label{enu:as2}measure space $\left(B,\mathscr{B},\mu\right)$ where
$B$ is a set equipped with a $\sigma$-algebra $\mathscr{B}$ of
subsets, and $\mu$ is a probability measure defined on $\mathscr{B}$. 

In particular, $\mu$ satisfies $\mu\left(\emptyset\right)=0$, $\mu\left(B\right)=1$,
$\mu\left(F\right)\geq0$ $\forall F\in\mathscr{B}$, and if $\left\{ F_{i}\right\} _{i\in\mathbb{N}}\subset\mathscr{B}$,
$F_{i}\cap F_{j}=\emptyset$, $i\neq j$ in $\mathbb{N}$, then $\mu\left(\cup_{i}F_{i}\right)=\sum_{i}\mu\left(F_{i}\right)$. 
\end{enumerate}
\begin{defn}
\label{def:a1}Let $K$ be a p.d. kernel as in (\ref{enu:as1}). We
shall denote by $\mathscr{H}\left(K\right)$ the corresponding \emph{reproducing
kernel Hilbert space} (RKHS), i.e., $\mathscr{H}\left(K\right)$ is
the Hilbert-completion of $span\left\{ K_{s}:=K\left(\cdot,s\right)\mathrel{;}s\in S\right\} $,
with respect to the inner product 
\begin{equation}
\left\langle \sum\xi_{i}K_{s_{i}},\sum\xi_{j}K_{s_{j}}\right\rangle _{\mathscr{H}\left(K\right)}:=\sum\sum\xi_{s_{i}}\overline{\xi_{s_{j}}}K\left(s_{i},s_{j}\right).\label{eq:pd3}
\end{equation}
The following reproducing property holds: 
\begin{equation}
f\left(s\right)=\left\langle f,K\left(\cdot,s\right)\right\rangle _{\mathscr{H}\left(K\right)},\;\forall s\in S,\;\forall f\in\mathscr{H}\left(K\right).\label{eq:rp1}
\end{equation}
\end{defn}

\begin{defn}
\label{def:dual}Given $K$ as in (\ref{enu:as1}), and $\left(B,\mu\right)$
as in (\ref{enu:as2}), we shall say that $\left(B,\mu\right)\in\mathcal{M}\left(K\right)$
if there is a function $k:S\longrightarrow L^{2}\left(\mu\right)$
such that 
\begin{equation}
K\left(s,t\right)=\int_{B}k_{s}\left(x\right)\overline{k_{t}\left(x\right)}d\mu\left(x\right)\label{eq:a2}
\end{equation}
holds for all $\left(s,t\right)\in S\times S$. We shall say that
$\left(B,\mu\right)$ is \emph{tight} (or minimal) iff the span of
$\left\{ k_{s}\mathrel{;}s\in S\right\} $ is dense in $L^{2}\left(B,\mu\right)$. 

Similarly, given $\left(B,\mu\right)$ as in (\ref{enu:as2}), we
shall say that a p.d. kernel $K$ is in $\mathscr{K}\left(\mu\right)$
if there is a stochastic process $\left\{ k_{s}\right\} _{s\in S}$
satisfying (\ref{eq:a2}).
\end{defn}

\begin{rem}
In \cite{2015arXiv150202549J}, we showed that for all p.d. kernel
$K\left(s,t\right)$, $\left(s,t\right)\in S\times S$, we have $\mathcal{M}\left(K\right)\neq\emptyset$.
See more examples below.
\end{rem}

Given $K$ as in (\ref{enu:as1}) then the problem (\ref{eq:a2})
always has a solution in a discrete (atomic) measure space relative
to the counting measure. Nonetheless, in the interesting solutions
$\left(B,\mathscr{B},\mu\right)$ to (\ref{eq:a2}) we aim to achieve
$B$ as a ``boundary space'' to the given set $S$ from (\ref{enu:as1});
see the details in \secref{HA} below.
\begin{defn}
We shall say that a Hilbert space $\mathscr{H}$ is \emph{separable}
if there is an orthonormal basis (ONB) $\left\{ \beta_{n}\right\} _{n\in\mathbb{N}}$
indexed by $\mathbb{N}$ (or a set of cardinality $\aleph_{0}$),
i.e., we have 
\begin{align}
\left\langle \beta_{n},\beta_{m}\right\rangle _{\mathscr{H}} & =\delta_{n,m},\;\text{and}\label{eq:a3}\\
\left\Vert f\right\Vert _{\mathscr{H}}^{2} & =\sum_{n\in\mathbb{N}}\left|\left\langle f,\beta_{n}\right\rangle _{\mathscr{H}}\right|^{2},\;\forall f\in\mathscr{H}.\label{eq:a4}
\end{align}

If only (\ref{eq:a4}) holds, we say that $\left\{ \beta_{n}\right\} _{n\in\mathbb{N}}$
is a \emph{Parseval frame}. In both cases, vectors $f$ in $\mathscr{H}$
always have the representation 
\begin{equation}
f=\sum_{n\in\mathbb{N}}\left\langle f,\beta_{n}\right\rangle _{\mathscr{H}}\beta_{n}\label{eq:a5}
\end{equation}
where (\ref{eq:a5}) converges in the norm $\left\Vert \cdot\right\Vert _{\mathscr{H}}$
of $\mathscr{H}$. 
\end{defn}

\begin{lem}
\label{lem:a1}Let $K$ be given and assumed positive definite (p.d.)
on $S\times S$, where $S$ is a set, see (\ref{enu:as1}). Let $\mathscr{H}=\mathscr{H}\left(K\right)$
be the corresponding reproducing kernel Hilbert space (RKHS), assumed
separable; see \defref{a1}. Let $\left\{ \beta_{n}\right\} _{n\in\mathbb{N}}$
be a Parseval frame, and set 
\begin{equation}
k_{s}\left(n\right):=\beta_{n}\left(s\right)=\left\langle \beta_{n},K\left(\cdot,s\right)\right\rangle _{\mathscr{H}};\label{eq:a6}
\end{equation}
see (\ref{eq:rp1}). Then the system (\ref{eq:a6}) is a solution
to (\ref{eq:a2}), but with the measure space $\mathbb{N}$, and with
counting measure.
\end{lem}

\begin{proof}
The existence of a Parseval frame $\left\{ \beta_{n}\right\} _{n\in\mathbb{N}}$
is assumed, so (\ref{eq:a3})-(\ref{eq:a4}) hold for the Hilbert
space $\mathscr{H}:=\mathscr{H}\left(K\right)$. Now, for all pairs
$\left(s,t\right)\in S\times S$, we have 
\begin{eqnarray*}
K\left(s,t\right) & = & \left\langle K\left(\cdot,s\right),K\left(\cdot,t\right)\right\rangle _{\mathscr{H}}\quad\left(\text{by the repoducing property}\right)\\
 & \underset{\text{by \ensuremath{\left(\ref{eq:pd3}\right)}}}{=} & \sum_{n\in\mathbb{N}}\left\langle K\left(\cdot,s\right),\beta_{n}\right\rangle _{\mathscr{H}}\left\langle \beta_{n},K\left(\cdot,t\right)\right\rangle _{\mathscr{H}}\\
 & \underset{\text{by \ensuremath{\left(\ref{eq:rp1}\right)}}}{=} & \sum_{n\in\mathbb{N}}\beta_{n}\left(s\right)\overline{\beta_{n}\left(t\right)}\\
 & \underset{\text{by \ensuremath{\left(\ref{eq:a6}\right)}}}{=} & \sum_{n\in\mathbb{N}}k_{s}\left(n\right)\overline{k_{t}\left(n\right)}
\end{eqnarray*}
which is the desired conclusion. 
\end{proof}
\begin{cor}
Let $K$ be given as in \lemref{a1} above, and let $\left\{ \beta_{n}\right\} _{n\in\mathbb{N}}$
be a Parseval frame, set $k_{s}\left(n\right):=\beta_{n}\left(s\right)$,
see (\ref{eq:a6}); then this is a \uline{minimal solution}, i.e.,
$\left\{ k_{s}\left(\cdot\right)\right\} $ is dense in $l^{2}\left(\mathbb{N}\right)$. 
\end{cor}

\begin{proof}
Immediate from the details in the proof of \lemref{a1}. In particular,
if $f\in\mathscr{H}=\mathscr{H}\left(K\right)$ satisfies $f\perp k_{s}\left(\cdot\right)$,
then $\left\langle f,\beta_{n}\left(\cdot\right)\right\rangle =0$
for all $n\in\mathbb{N}$, so by (\ref{eq:a4}) we get 
\[
\left\Vert f\right\Vert _{\mathscr{H}}^{2}=\sum_{n\in\mathbb{N}}\left|\left\langle f,\beta_{n}\right\rangle _{\mathscr{H}}\right|^{2}=0.
\]
\end{proof}
\textbf{Discussion of the literature.} The theory of RKHS and their
applications is vast, and below we only make a selection. Readers
will be able to find more cited there. As for the general theory of
RKHS in the pointwise category, we find useful \cite{AD92,ABDdS93,AD93,MR2529882,MR3526117}.
The applications include fractals (see e.g., \cite{MR0008639,AJSV13,MR3267010});
probability theory \cite{MR0277027,MR3379106,MR3504608,MR3624688,MR3571702,MR3601658,MR3622604};
and application to learning theory \cite{MR2058288,MR2558684,MR2488871,MR3450534}.
For recent applications, we refer to \cite{MR3450534,MR3441738,MR3559001}. 

\section{Properties of solutions to the factorization problem}

Let $K:S\times S\longrightarrow\mathbb{C}$ be a given p.d. kernel,
as specified in (\ref{enu:as1}) from \secref{Intro} above. Solutions
$\left(B,\mu,\left\{ k_{s}\right\} _{s\in S}\right)$ to the problem
(\ref{eq:a2}) are called \emph{factorizations}.
\begin{prop}
Let $K$ on $S\times S$ be given, and let $\left(B,\mu,\left\{ k_{s}\right\} _{s\in S}\right)$
be a solution to the factorization problem (\ref{eq:a2}). Then the
assignment 
\begin{equation}
W\left(K\left(\cdot,s\right)\right):=k_{s}\in L^{2}\left(\mu\right)\label{eq:e1}
\end{equation}
extends by linearity to an isometry, denoted $W:\mathscr{H}\left(K\right)\longrightarrow L^{2}\left(\mu\right)$,
and its adjoint $V:=W^{*}:L^{2}\left(\mu\right)\longrightarrow\mathscr{H}\left(K\right)$
is the following transform
\begin{equation}
\left(Vf\right)\left(s\right)=\int_{B}f\left(x\right)\overline{k_{s}\left(x\right)}d\mu\left(x\right),\label{eq:e2}
\end{equation}
and we have
\begin{equation}
W^{*}W=VW=I_{\mathscr{H}\left(K\right)},\label{eq:e3}
\end{equation}
while $WW^{*}=WV$ is a projection in the Hilbert space $L^{2}\left(\mu\right)$.
\end{prop}

\begin{proof}
It is immediate from (\ref{eq:a2}) that the operator $W$ from (\ref{eq:e1})
is isometric $\mathscr{H}\left(K\right)\longrightarrow L^{2}\left(\mu\right)$.
For $f\in L^{2}\left(\mu\right)$, and $s\in S$, we have
\begin{align*}
\left\langle W\left(K\left(\cdot,s\right)\right),f\right\rangle _{L^{2}\left(\mu\right)} & =\int_{B}k_{s}\left(x\right)\overline{f\left(x\right)}d\mu\left(x\right)\\
 & =\left\langle K\left(\cdot,s\right),W^{*}f\right\rangle _{\mathscr{H}\left(K\right)}=\overline{\left(W^{*}f\right)\left(s\right)},
\end{align*}
so the formula (\ref{eq:e2}) follows, and we infer that 
\begin{equation}
\Big(s\longmapsto\left(Vf\right)\left(s\right)=\int_{B}f\left(x\right)\overline{k_{s}\left(x\right)}d\mu\left(x\right)\Big)\in\mathscr{H}\left(K\right).\label{eq:e4}
\end{equation}

We will show that for $\forall N\in\mathbb{N}$, $\forall\left\{ s_{i}\right\} _{i=1}^{N}$,
$s_{i}\in S$, $\left\{ \xi_{i}\right\} _{i=1}^{N}$, $\xi_{i}\in\mathbb{C}$,
we have
\begin{equation}
\left|\sum_{i=1}^{N}\xi_{i}\left(Vf\right)\left(s_{i}\right)\right|^{2}\leq\left\Vert f\right\Vert _{L^{2}\left(\mu\right)}^{2}\sum_{i=1}^{N}\sum_{j=1}^{N}\xi_{i}\overline{\xi}_{j}K\left(s_{i},s_{j}\right),\label{eq:e5}
\end{equation}
and the desired conclusion (\ref{eq:e4}) follows. We now show that
(\ref{eq:e5}) holds: 
\begin{align*}
\text{LHS}_{\left(\ref{eq:e5}\right)} & =\left|\int_{B}f\left(x\right)\overline{\sum_{i}\xi_{i}k_{s_{i}}\left(x\right)}d\mu\left(x\right)\right|^{2}\\
 & \leq\int_{B}\left|f\left(x\right)\right|^{2}d\mu\left(x\right)\int_{B}\left|\sum_{i}\xi_{i}k_{s_{i}}\left(x\right)\right|^{2}d\mu\left(x\right)\quad\left(\text{by Schwarz in }L^{2}\left(\mu\right)\right)\\
 & =\left\Vert f\right\Vert _{L^{2}\left(\mu\right)}^{2}\sum_{i}\sum_{j}\xi_{i}\overline{\xi}_{j}\int_{B}k_{s_{i}}\left(x\right)\overline{k_{s_{j}}\left(x\right)}d\mu\left(x\right)\\
 & =\left\Vert f\right\Vert _{L^{2}\left(\mu\right)}^{2}\sum_{i}\sum_{j}\xi_{i}\overline{\xi}_{j}K\left(s_{i},s_{j}\right)=\text{RHS}_{\left(\ref{eq:e5}\right)},
\end{align*}
and the proof is completed. 
\end{proof}

\section{The boundary space}

Given $\left(K,S\right)$ as in (\ref{enu:as1}), i.e., $S$ is a
set and $K$ is a p.d. kernel on $S\times S$, let $\mathcal{M}\left(K\right)$
be the boundary space consisting of all measure spaces $\left(B,\mu\right)$
satisfying (\ref{eq:a2}); see \defref{dual}. 

In the discussion below, we shall introduce an order relation on $\mathcal{M}\left(K\right)$.
We show that there is always a minimal element in $\mathcal{M}\left(K\right)$. 
\begin{defn}
\label{def:po}Suppose $\left(B_{i},\mathscr{B}_{i},\mu_{i}\right)\in\mathcal{M}\left(K\right)$,
$i=1,2$. We say that 
\begin{equation}
\left(B_{1},\mathscr{B}_{1},\mu_{1}\right)\leq\left(B_{2},\mathscr{B}_{2},\mu_{2}\right)\label{eq:mo0}
\end{equation}
if $\exists\varphi:B_{2}\longrightarrow B_{1}$, s.t. 
\begin{align}
\mu_{2}\circ\varphi^{-1} & =\mu_{1},\;\mbox{and}\label{eq:mo1}\\
\varphi^{-1}\left(\mathscr{B}_{1}\right) & =\mathscr{B}_{2}.\label{eq:mo2}
\end{align}
\end{defn}

\begin{lem}
\label{lem:me} $\mathcal{M}\left(K\right)$ has minimal elements.
\end{lem}

\begin{proof}
If (\ref{eq:mo1})-(\ref{eq:mo2}) hold, then 
\[
L^{2}\left(B_{1},\mu_{1}\right)\ni f\xrightarrow{\quad W_{21}}f\circ\varphi\in L^{2}\left(B_{2},\mu_{2}\right)
\]
is isometric, i.e., 
\begin{equation}
\int_{B_{2}}\underset{W_{21}f}{|\underbrace{f\circ\varphi}|^{2}}d\mu_{2}=\int_{B_{1}}\left|f\right|^{2}d\mu_{1},
\end{equation}
and
\begin{equation}
W_{B_{2}}=W_{21}W_{B_{1}}\;\mbox{on }\mathscr{H}\left(K\right),
\end{equation}
i.e., the diagram commutes:
\[
\xymatrix{ &  & L^{2}\left(B_{1},\mu_{1}\right)\ar[d]^{W_{21}}\\
\mathscr{H}\left(K\right)\ar@/^{1.3pc}/[rru]^{W_{B_{1}}}\ar[rr]_{W_{B_{2}}} &  & L^{2}\left(B_{2},\mu_{2}\right)
}
\]

We can then use Zorn's lemma to prove that $\forall K$, $\mathcal{M}\left(K\right)$
has minimal elements $\left(B,\mathscr{B},\mu\right)$. (See the proof
of \thmref{KS} below.) But even if $\left(B,\mathscr{B},\mu\right)$
is minimal, $W_{B}:\mathscr{H}\left(K\right)\rightarrow L^{2}\left(\mu\right)$
may \emph{not} be onto.
\end{proof}
In the next result, we shall refer to the partial order ``$\leq$''
from (\ref{eq:mo0}) when considering minimal elements in $\mathcal{M}\left(K\right)$.
And, in referring to $\mathcal{M}\left(K\right)$, we have in mind
a fixed positive definite function $K:S\times S\rightarrow\mathbb{C}$,
specified at the outset; see (\ref{eq:a1}).
\begin{thm}
\label{thm:KS}Let $\left(K,S\right)$ be a fixed positive definite
kernel, and let $\mathcal{M}\left(K\right)$ be the corresponding
boundary space from \defref{dual}. 

Then, for every $\left(X,\lambda\right)\in\mathcal{M}\left(K\right)$,
there is a $\left(M,\nu\right)\in\mathcal{M}\left(K\right)$ such
that 
\begin{equation}
\left(M,\nu\right)\leq\left(X,\lambda\right),\label{eq:pa1}
\end{equation}
and $\left(M,\nu\right)$ is \uline{minimal} in the following sense:
Suppose $\left(B,\mu\right)\in\mathcal{M}\left(K\right)$ and 
\begin{equation}
\left(B,\mu\right)\leq\left(M,\nu\right),\label{eq:pa2}
\end{equation}
then it follows that $\left(B,\mu\right)\simeq\left(M,\nu\right)$,
i.e., we also have $\left(M,\nu\right)\leq\left(B,\mu\right)$. 
\end{thm}

\begin{proof}
We shall use Zorn's lemma, and the argument from \lemref{me}.

Let $\mathcal{L}=\left\{ \left(B,\mu\right)\right\} $ be a linearly
ordered subset of $\mathcal{M}\left(K\right)$ s.t. 
\begin{equation}
\left(B,\mu\right)\leq\left(X,\lambda\right),\quad\forall\left(B,\mu\right)\in\mathcal{L};\label{eq:pa3}
\end{equation}
and such that, for every pair $\left(B_{i},\mu_{i}\right)$, $i=1,2$,
in $\mathcal{L}$, one of the following two cases must hold:
\begin{equation}
\left(B_{1},\mu_{1}\right)\leq\left(B_{2},\mu_{2}\right),\;\mbox{or }\left(B_{2},\mu_{2}\right)\leq\left(B_{1},\mu_{1}\right).\label{eq:pa4}
\end{equation}
To apply Zorn's lemma, we must show that there is a $\left(B_{\mathcal{L}},\mu_{\mathcal{L}}\right)\in\mathcal{M}\left(K\right)$
such that 
\begin{equation}
\left(B_{\mathcal{L}},\mu_{\mathcal{L}}\right)\leq\left(B,\mu\right),\quad\forall\left(B,\mu\right)\in\mathcal{L}.\label{eq:pa5}
\end{equation}

Now, using (\ref{eq:pa3})-(\ref{eq:pa4}), we conclude that the measure
spaces $\left\{ \left(B,\mu\right)\right\} _{\mathcal{L}}$ have an
inductive limit, i.e., the existence of:
\begin{equation}
\mu_{\mathcal{L}}:=\underset{B\xrightarrow[\;\mathcal{L}\;]{}B_{\mathcal{L}}}{\mbox{ind limit }}\mu_{B}.\label{eq:pa6}
\end{equation}
In other words, we may apply Kolmogorov's consistency (see, e.g.,
\cite{PaSc75}) to the family $\mathcal{L}$ of measure spaces in
order to justify the inductive limit construction in (\ref{eq:pa6}).

We have proved that every linearly ordered subset $\mathcal{L}$ (as
specified) has a ``lower bound'' in the sense of (\ref{eq:pa5}).
Hence Zorn's lemma applies, and the desired conclusion follows, i.e.,
there is a pair $\left(M,\nu\right)\in\mathcal{M}\left(K\right)$
which satisfies the condition (\ref{eq:pa2}) from the theorem. 
\end{proof}

\section{Gaussian processes}

By a theorem of Kolmogorov, every Hilbert space may be realized as
a (Gaussian) reproducing kernel Hilbert space (RKHS), see e.g., \cite{IM65,PaSc75,NF10},
and \thmref{g1} below.
\begin{rem}
~
\begin{enumerate}
\item \label{enu:ev1}Given a positive definite (p.d.) kernel $K$ on $S\times S$,
there is then an associated mapping $E_{S}:S\rightarrow\left\{ \mbox{Functions on }S\right\} $
given by 
\begin{equation}
E_{S}\left(t\right)=K\left(t,\cdot\right),\label{eq:ev1}
\end{equation}
where the dot ``$\cdot$'' in (\ref{eq:ev1}) indicates the independent
variable; so 
\[
S\ni s\longrightarrow K\left(t,s\right)\in\mathbb{C}.
\]
 
\item \label{enu:ev2}We shall assume that $E_{S}$ is 1-1, i.e., if $s_{1},s_{2}\in S$,
and $K\left(s_{1},t\right)=K\left(s_{2},t\right)$, $\forall t\in S$,
then it follows that $s_{1}=s_{2}$. (This is not a strong limiting
condition on $K$.) 
\item We shall view the Cartesian product 
\begin{equation}
B_{S}:=\prod_{S}\mathbb{C}=\mathbb{C}^{S}\label{eq:ev2}
\end{equation}
as the set of all functions $S\rightarrow\mathbb{C}$. 
\end{enumerate}
It follows from assumption (\ref{enu:ev2}) that $E_{S}:S\rightarrow B_{S}$
is an injection, i.e., with $E_{S}$, we may identity $S$ as a ``subset''
of $B_{S}$. 

For $v\in S$, set $\pi_{v}:B_{S}\longrightarrow\mathbb{C}$, 
\begin{equation}
\pi_{v}\left(x\right)=x\left(v\right),\quad\forall x\in B_{S};\label{eq:ev3}
\end{equation}
i.e., $\pi_{v}$ is the coordinate mapping at $v$. The topology on
$B_{S}$ shall be the product topology; and similarly the $\sigma$-algebra
$\mathscr{B}_{S}$ will be the the one generated by $\left\{ \pi_{v}\right\} _{v\in S}$,
i.e., generated by the family of subsets
\begin{equation}
\pi_{v}^{-1}\left(M\right),\;v\in S,\;\mbox{and }M\subset\mathbb{C}\;\mbox{a Borel set}.\label{eq:ev4}
\end{equation}

\end{rem}

\begin{thm}[Every p.d. kernel has a (non-minimal) Gaussian solution]
\label{thm:g1}Let $\left(S,K\right)$ be as specified in (\ref{enu:as1}),
then there is a Gaussian solution $\left(B,\mathscr{B},\mu,\left\{ k_{s}\right\} _{s\in S}\right)$
to (\ref{eq:a2}). 
\end{thm}

By a Gaussian solution we mean $\left(B,\mathscr{B},\mu\right)$ is
a probability space, and $k:S\longrightarrow L^{2}\left(\mu\right)$
has the following properties: 
\begin{enumerate}[label=(\alph{enumi}),ref=\alph{enumi}]
\item \label{enu:ga}Condition (\ref{eq:a2}) holds. We shall write $K\left(s,t\right)=\mathbb{E}\left(k_{s}\overline{k_{t}}\right)$
where $\mathbb{E}$ denotes the expectation with respect to $\mu$;
\item $\mathbb{E}\left(k_{s}\right)=0$, $\forall s\in S$; 
\item \label{enu:gc}For every finite subset $F\subset S$, the system of
random variables $\left\{ k_{s}\right\} _{s\in F}$ is jointly Gaussian
with covariance matrix $M_{F}$ given by
\begin{equation}
M_{F}\left(s,t\right)=K\left(s,t\right),\;\forall\left(s,t\right)\in F\times F.\label{eq:c3}
\end{equation}
\end{enumerate}
\begin{proof}[Proof of \thmref{g1} (sketch)]
 The result is essentially an application of the Kolmogorov extension
principle (see e.g., \cite{PaSc75}): Take 
\begin{equation}
B:=\prod_{s\in S}\mathbb{C}=\text{all functions on }S,\label{eq:c4}
\end{equation}
and set 
\begin{equation}
k_{s}\left(x\right)=x\left(s\right),\;\forall s\in S.\label{eq:c5}
\end{equation}

Let $F\subset S$ be a fixed subset, and let $\mu_{F}$ be the Gaussian
measure on $\mathbb{C}^{F}$ which is specified by zero mean, and
covariance matrix $M_{F}$ as in (\ref{eq:c3}). If $M_{F}$ is invertible,
then the density on $\mathbb{C}^{F}$ computed w.r.t. Lebesgue measure
on $\mathbb{R}^{2\left|F\right|}$ is 
\begin{equation}
\det\left(M_{F}\right)^{-\left|F\right|}\exp\left(-\frac{1}{2}\left\langle M_{F}^{-1}z_{F},z_{F}\right\rangle _{l^{2}\left(F\right)}\right)\label{eq:c6}
\end{equation}
where $z_{F}$ denotes the point in $l^{2}\left(F\right)$ with components
$z_{j}\in\mathbb{C}$ now indexed by $j\in F$. 

The system of measures $\left\{ \mu_{F}\right\} $ induced by all
finite subsets of $S$ then satisfies the Kolmogorov consistency equation:
If $F\subset F'$ are two finite subsets, then
\begin{equation}
\mathbb{E}\left(\mu_{F'}\mid\mathbb{C}^{F}\right)=\mu_{F}\label{eq:c7}
\end{equation}
where the notation in (\ref{eq:c7}) refers to the conditional measure,
and $\mathbb{C}^{F}\hookrightarrow\mathbb{C}^{F'}$ via 
\begin{equation}
\mathbb{C}^{F'}=\mathbb{C}^{F}\times\mathbb{C}^{F'\backslash F}.\label{eq:c8}
\end{equation}

The existence of the desired probability measure $\mu$ on $\mathscr{B}$
now follows from Kolmogorov's theorem, and we automatically get 
\begin{equation}
\mathbb{E}\left(\mu\mid\mathbb{C}^{F}\right)=\mu_{F}\label{eq:c9}
\end{equation}
valid for all finite subsets $F\subset S$. Now (\ref{eq:c9}) refers
to conditioning via $B=\mathbb{C}^{F}\times\mathbb{C}^{S\backslash F}$. 

The stated conditions (\ref{enu:ga})-(\ref{enu:gc}) therefore follow,
and the process $\left\{ k_{s}\right\} _{s\in S}$ in (\ref{eq:c5})
has the desired properties. 
\end{proof}
\begin{defn}
Given $K$ on $S\times S$ p.d. as in (\ref{enu:as1}) from \secref{Intro}.
A solution to (\ref{eq:a2}) (see (\ref{enu:as2})), $\left(B,\mathscr{B},\mu,\left\{ k_{s}\right\} _{s\in S}\right)$,
is said to be \emph{minimal} iff (Def.) the $L^{2}\left(\mu\right)$
closure of $span\left\{ k_{s}\mathrel{;}s\in S\right\} $ is all of
$L^{2}\left(\mu\right)$, i.e., $\overline{span}{}^{L^{2}\left(\mu\right)}\left\{ k_{s}\right\} _{s\in S}=L^{2}\left(\mu\right)$.
\end{defn}

\begin{rem}
It is known that the solution from \thmref{g1} is generally \emph{not}
minimal; see e.g., \cite{MR0277027,PaSc75,AD92,MR2483792,AJSV13,MR3330163,MR3447224,MR3526117}.

Indeed, given $K$ on $S\times S$, p.d. as specified in (\ref{enu:as1}),
let $\left(B,\mathscr{B},\mu,\left\{ k_{s}\right\} _{s\in S}\right)$
be the Gaussian solution from \thmref{g1}; then $L^{2}\left(B,\mu\right)$
is isomorphic to the symmetric Fock space $\mathscr{F}_{s}\left(\mathscr{H}_{1}\right)$
where $\mathscr{H}_{1}=\overline{span}\left\{ k_{s}\mathrel{;}s\in S\right\} $. 
\end{rem}

\begin{example}[A p.d. kernel (the Szeg\H{o} kernel) with minimal solutions]
Let $\mathbb{D}:=\left\{ z\in\mathbb{C}\mathrel{;}\left|z\right|<1\right\} $,
the open disk in the complex plane $\mathbb{C}$, and let 
\begin{equation}
K\left(z,w\right)=\frac{1}{1-z\overline{w}},\;\left(z,w\right)\in\mathbb{D}\times\mathbb{D},\label{eq:d1}
\end{equation}
be the Szeg\H{o} kernel (see e.g., \cite{MR3526117}). Let $\mu$
be a singular measure on $\mathbb{T}=\partial\mathbb{D}\simeq\left[0,1\right]$.
We use the isomorphism $\left[0,1\right]\simeq\mathbb{T}$, given
by $\left[0,1\right]\ni x\longmapsto e\left(x\right)=e^{i2\pi x}\in\mathbb{T}$.
In this case, take 
\begin{equation}
k_{z}\left(x\right)=\frac{1}{1-z\overline{e\left(x\right)}},\;x\in\left[0,1\right],\label{eq:d2}
\end{equation}
and suppose $f\in L^{2}\left(\left[0,1\right],\mu\right)$ satisfies
\begin{equation}
\left\langle f,k_{z}\right\rangle _{L^{2}\left(\mu\right)}=0,\;\forall z\in\mathbb{D}.
\end{equation}
Hence, 
\begin{equation}
\int_{0}^{1}e\left(nx\right)f\left(x\right)d\mu\left(x\right)=0,\;\forall n\in\mathbb{N}_{0}.
\end{equation}
By the F. \& M. Riesz theorem, we conclude that $fd\mu\ll dx$ holds,
where $dx$ is standard Lebesgue measure. Since $fd\mu\perp dx$ by
assumption, we conclude that $f=0$ in $L^{2}\left(\mu\right)$. 
\end{example}

\begin{thm}[\cite{2016arXiv160308852H}]
 Let $\mu$ be a \uline{singular} probability measure on $\left[0,1\right]$,
and set 
\begin{equation}
b\left(z\right):=1-\int_{0}^{1}\frac{d\mu\left(x\right)}{1-z\overline{e\left(x\right)}},
\end{equation}
see (\ref{eq:d2}), and 
\begin{equation}
K^{\left(b\right)}\left(z,w\right)=\frac{1-b\left(z\right)\overline{b\left(w\right)}}{1-z\overline{w}},\;\left(z,w\right)\in\mathbb{D}\times\mathbb{D},\label{eq:d3}
\end{equation}
see (\ref{eq:d1}); then $\mu\in\mathcal{M}\left(K^{\left(b\right)}\right)$;
and it is a minimal solution, i.e., 
\[
k_{z}^{\left(b\right)}=\frac{1-b\left(z\right)\overline{b\left(e\left(x\right)\right)}}{1-z\overline{e\left(x\right)}}
\]
satisfies 
\begin{equation}
K^{b}\left(z,w\right)=\int_{0}^{1}k_{z}^{\left(b\right)}\left(x\right)\overline{k_{w}^{\left(b\right)}\left(x\right)}d\mu\left(x\right),\;\left(\text{see \ensuremath{\left(\ref{eq:a2}\right)}}\right)
\end{equation}
and $\{k_{z}^{\left(b\right)}\left(x\right)\}_{z\in\mathbb{D}}$ spans
a dense subspace in $L^{2}\left(\mu\right)$.

Moreover, $b$ is an inner function, i.e., $b\in H^{\infty}$, with
boundary values $\left|b\left(e\left(x\right)\right)\right|=1$ a.e.
$x$. For the RKHS of $K^{\left(b\right)}$ in (\ref{eq:d3}), we
have 
\begin{equation}
\mathscr{H}(K^{\left(b\right)})=H^{2}\ominus bH^{2}
\end{equation}
where $H^{2}$ is the standard Hardy space on $\mathbb{D}$. 
\end{thm}

\section{\label{sec:HA}Harmonic analysis}

In a general setting, positive definite (p.d.) kernels $K$ are defined
on $S\times S$ where $S$ is a fixed set. In classical analysis such
pairs $\left(K,S\right)$ have found uses in many problems in harmonic
analysis, in complex analysis, in stochastic analysis, analysis on
infinite graphs, and in PDE theory, the latter in the context of Green's
functions for elliptic operators. In the complex analysis setting,
$S$ may be the disk $\mathbb{D}$, or the upper half-plane. For these
applications, solutions typically entail consideration of boundaries,
some in a natural geometric framework, and some more abstract. In
some of the applications considered here, the notion of \textquotedblleft boundary\textquotedblright{}
is clear enough, for example for real or complex domains, but not
for others. Take for example the case when $S$ may instead be the
set of vertices in an infinite graph.

The problem considered in the present section is motivated by p.d.
kernels arising naturally from classical frameworks, but our emphasis
will be applications when there is not already a given, or a natural
boundary available at the outset. 
\begin{example}
\label{exa:pdisk}Let $S:=\mathbb{D}^{k}$ (the polydisk) with boundary
$B:=\mathbb{T}^{k}\simeq I^{k}$, and 
\begin{equation}
K=\prod_{j=1}^{k}\left(\frac{1}{1-z_{j}\overline{w}_{j}}\right).\label{eq:b1}
\end{equation}
Recall the multi-index notation: $z=\left(z_{1},\cdots,z_{k}\right)$,
$w=\left(w_{1},\cdots,w_{k}\right)$ in $\mathbb{D}^{k}$; and $z^{n}=z_{1}^{n_{1}}\cdots z_{k}^{n_{k}}$,
$n=\left(n_{1},\cdots,n_{k}\right)\in\mathbb{N}_{0}^{k}$. 
\end{example}

\textbf{General setting.} Given $S$ a set, $K$ a p.d. kernel on
$S$, $\left(B,\mathscr{B},\mu\right)$ a probability space, and $K^{*}:S\times B\longrightarrow\mathbb{C}$,
assume that $\mu\in\mathcal{M}\left(K\right)$, $K\in\mathscr{K}\left(\mu\right)$
with reference to $S\longleftrightarrow B$. That is, 
\begin{equation}
K\left(z,w\right)=\int_{B}K_{z}^{*}\left(x\right)\overline{K_{w}^{*}\left(x\right)}d\mu\left(x\right),\;\forall\left(z,w\right)\in S\times S.\label{eq:b2}
\end{equation}
Set 
\[
\mathbb{E}\left(K_{z}^{*}\right)=\int_{B}K_{z}^{*}\left(x\right)d\mu\left(x\right),\;\forall z\in S,
\]
and 
\begin{equation}
K^{ren}\left(z,w\right)=\frac{1}{\mathbb{E}\left(K_{z}^{*}\right)\mathbb{E}\left(\overline{K_{w}^{*}}\right)}K\left(z,w\right),\;\forall\left(z,w\right)\in S\times S,\label{eq:b3}
\end{equation}
where ``ren'' $:=$ renormalization. 

The kernel $K^{ren}$ in (\ref{eq:b3}) is p.d., and we shall denote
the corresponding RKHS $\mathscr{H}^{ren}=\mathscr{H}\left(K^{ren}\right)$. 

Set 
\begin{equation}
\left(K_{z}^{ren}\right)^{*}\left(x\right):=\frac{K_{z}^{*}\left(x\right)}{\mathbb{E}\left(K_{z}^{*}\right)},\;\forall z\in S.\label{eq:b4}
\end{equation}
\begin{lem}
We have 
\begin{equation}
K^{ren}\left(z,w\right)=\int_{B}\left(K_{z}^{ren}\right)^{*}\left(x\right)\overline{\left(K_{w}^{ren}\right)^{*}\left(x\right)}d\mu\left(x\right),\label{eq:b5}
\end{equation}
on $S\times S$. 
\end{lem}

\begin{proof}
Note that 
\begin{eqnarray*}
\text{RHS}_{\left(\ref{eq:b5}\right)} & \underset{\text{by }\left(\ref{eq:b4}\right)}{=} & \frac{1}{\mathbb{E}\left(K_{z}^{*}\right)\mathbb{E}\left(\overline{K_{w}^{*}}\right)}\int_{B}K_{z}^{*}\left(x\right)\overline{K_{w}^{*}\left(x\right)}d\mu\left(x\right)\\
 & \underset{\text{by }\left(\ref{eq:b2}\right)}{=} & \frac{K\left(z,w\right)}{\mathbb{E}\left(K_{z}^{*}\right)\mathbb{E}\left(\overline{K_{w}^{*}}\right)}=K^{ren}\left(z,w\right)=\text{LHS}_{\left(\ref{eq:b5}\right)}.
\end{eqnarray*}
\end{proof}
\begin{defn}
Let $S$, $B$, $\mu$, $K$, $K^{ren}$, and $K^{*}$ etc. be as
above. Set 
\begin{equation}
K_{z}^{ren}\xmapsto{\quad W_{\mu}\quad}\left(K_{z}^{ren}\right)^{*}\left(x\right)\left(=\frac{K_{z}^{*}\left(x\right)}{\mathbb{E}\left(K_{z}^{*}\right)}\right).\label{eq:b6}
\end{equation}
The assignment (\ref{eq:b6}) extends by limit and closure to an isometry
\begin{equation}
W_{\mu}:\mathscr{H}^{ren}\longrightarrow L^{2}\left(\mu\right),\label{eq:b7}
\end{equation}
so that $I_{\mathscr{H}^{ren}}=W_{\mu}^{*}W_{\mu}$. 
\end{defn}

\begin{rem}
(\ref{eq:b7}) is immediate from (\ref{eq:b4})-(\ref{eq:b5}). In
general, $W_{\mu}$ may not be onto; see below.
\end{rem}

\begin{lem}
The adjoint operator $V_{\mu}:=W_{\mu}^{*}$ of the isometry in (\ref{eq:b7})
is a co-isometry, determined as follows: 
\begin{equation}
\xymatrix{\mathscr{H}^{ren}\ar@/^{1.3pc}/[rr]^{W_{\mu}} &  & L^{2}\left(\mu\right)\ar@/^{1.3pc}/[ll]^{V_{\mu}}}
\label{eq:b7a}
\end{equation}
For $f\in L^{2}\left(\mu\right)$, $z\in S$, we have
\begin{align}
\left(V_{\mu}f\right)\left(z\right) & =\int_{B}f\left(z\right)\overline{\left(\frac{K_{z}^{*}}{\mathbb{E}\left(K_{z}^{*}\right)}\right)}d\mu\left(x\right)\nonumber \\
 & =\frac{1}{\mathbb{E}\left(\overline{K_{z}^{*}}\right)}\int_{B}f\left(x\right)\overline{K_{z}^{*}\left(x\right)}d\mu\left(x\right).\label{eq:b8}
\end{align}
(We call $V_{\mu}$ a normalized transform.)
\end{lem}

\begin{proof}
Immediate from the definitions. Indeed, for $\forall f\in L^{2}\left(\mu\right)$,
$z\in S$, we have 
\[
\left\langle V_{\mu}f,K_{z}^{ren}\right\rangle _{\mathscr{H}^{ren}}=\int_{B}f\left(x\right)\frac{\overline{K_{z}^{*}\left(x\right)}}{\mathbb{E}\left(\overline{K_{z}^{*}}\right)}d\mu\left(x\right)=\left\langle f,W_{\mu}K_{z}^{ren}\right\rangle _{L^{2}\left(\mu\right)}.
\]
The result follows, $V^{*}=W_{\mu}$, $W_{\mu}^{*}=V_{\mu}$. 
\end{proof}
\begin{cor}
The co-isometry $V_{\mu}:L^{2}\left(\mu\right)\longrightarrow\mathscr{H}^{ren}$
is defined on all of $L^{2}\left(\mu\right)$ if and only if
\begin{equation}
\overline{span}\left\{ K_{z}^{*}\left(\cdot\right)\mathrel{;}z\in S\right\} =L^{2}\left(\mu\right),\label{eq:b9}
\end{equation}
where the LHS of (\ref{eq:b9}) denotes the $L^{2}\left(\mu\right)$-closed
span of $\left\{ K_{z}^{*}\left(\cdot\right)\mathrel{;}z\in S\right\} $. 
\end{cor}

\begin{proof}
Recall that $ran\left(W_{\mu}\right)^{\perp}=ker\left(V_{\mu}\right)$,
so $ker\left(V_{\mu}\right)=0$ $\Longleftrightarrow$ (\ref{eq:b9})
holds. Note that when (\ref{eq:b9}) holds then we get a unitary isomorphism
$V_{\mu}=W_{\mu}^{*}$. See (\ref{eq:b7a}).
\end{proof}
Return to the polydisk in \exaref{pdisk}. 
\begin{lem}
Equation (\ref{eq:b9}) holds in this special case, i.e., 
\begin{gather}
\overline{span}\left\{ K_{z}^{*}\left(\cdot\right)\right\} _{x\in\mathbb{D}^{k}}=L^{2}\left(\mu\right)\nonumber \\
\Updownarrow\nonumber \\
span\left\{ e_{n_{1}}\left(x_{1}\right)e_{n_{2}}\left(x_{2}\right)\cdots e_{n_{k}}\left(x_{k}\right)\mathrel{;}n\in\mathbb{N}_{0}^{k}\right\} \;\text{is dense in }L^{2}\left(I^{k},\mu\right).\label{eq:b11}
\end{gather}
\end{lem}

\begin{proof}
(\ref{eq:b11}) follows from the orthogonality relation in $L^{2}\left(\mu\right)$,
\begin{equation}
\left\{ K_{z}^{*}\left(\cdot\right)\right\} _{z\in\mathbb{D}^{k}}^{\perp}=\left\{ e_{n}\left(\cdot\right)\right\} _{n\in\mathbb{N}_{0}^{k}}^{\perp}\label{eq:b12}
\end{equation}
where $\perp$ refers to the $L^{2}\left(\mu\right)$-inner product.
In details, we have
\begin{gather*}
f\perp K_{z}^{*},\;\forall z\in\mathbb{D}^{k},\\
\Updownarrow\\
\mathbb{D}^{k}\ni z\longmapsto\left\langle K_{z}^{*},f\right\rangle _{\mu}\equiv0,\\
\Updownarrow\\
\left(\frac{\partial}{\partial z}\right)^{n}\left\langle K_{z}^{*},f\right\rangle _{\mu}=0,\;\forall n\in\mathbb{N}_{0}^{k},\\
\Updownarrow\\
\left\langle e_{n},f\right\rangle _{\mu}=0,\;\forall n\in\mathbb{N}_{0}^{k},
\end{gather*}
which is the desired conclusion. 
\end{proof}
In the case of polydisk, we know that $1/\mathbb{E}\left(K_{z}^{*}\right)$
is a multiplier in $\mathscr{H}\left(K\right)$, and so 
\begin{equation}
\mathscr{H}^{ren}\hookrightarrow\mathscr{H}\label{eq:b14}
\end{equation}
where (\ref{eq:b14}) means containment of RKHSs, i.e.,

$\mathscr{H}^{ren}=\mathscr{H}\left(K^{ren}\right)=$ the RKHS of
$K^{ren}$; see (\ref{eq:b3})-(\ref{eq:b5}). 

$\mathscr{H}=\mathscr{H}\left(K\right)=$ the RKHS of the Szeg\H{o}
kernel (\ref{eq:b1}). 
\begin{example}
For $k=1$, $1/\mathbb{E}\left(K_{z}^{*}\right)=1-b\left(z\right)$,
where $b$ is the function corresponding to $\mu$ via Herglotz 
\begin{equation}
\Re\left\{ \frac{1+b\left(z\right)}{1-b\left(z\right)}\right\} =P_{z}\left[\mu\right]=\int_{0}^{1}\frac{1-\left|z\right|^{2}}{\left|e\left(x\right)-z\right|^{2}}d\mu\left(x\right)\label{eq:b15}
\end{equation}
i.e., $P_{z}\left[\mu\right]$ is the Poisson-kernel integral. 
\end{example}

\begin{rem}
In the \emph{general case }discussed above, we may not have that $S\ni z\longmapsto1/\mathbb{E}\left(K_{z}^{*}\right)$
is a multiplier in $\mathscr{H}\left(K\right)$. (See (\ref{eq:b2})-(\ref{eq:b3})
for the definitions.) It may not even be so for all kernels on $\mathbb{D}^{k}$. 
\end{rem}

\begin{acknowledgement*}
The co-authors thank the following colleagues for helpful and enlightening
discussions: Professors Sergii Bezuglyi, Ilwoo Cho, Paul Muhly, Myung-Sin
Song, Wayne Polyzou, and members in the Math Physics seminar at The
University of Iowa.
\end{acknowledgement*}
\bibliographystyle{amsalpha}
\bibliography{ref}

\end{document}